\documentclass[pdflatex,sn-mathphys-num]{sn-jnl}

\usepackage{graphicx}%
\usepackage{multirow}%
\usepackage{amsmath,amssymb,amsfonts}%
\usepackage{amsthm}%
\usepackage{mathrsfs}%
\usepackage[title]{appendix}%
\usepackage{xcolor}%
\usepackage{textcomp}%
\usepackage{manyfoot}%
\usepackage{booktabs}%
\usepackage{algorithm}%
\usepackage{algorithmicx}%
\usepackage{algpseudocode}%
\usepackage{listings}%
\usepackage{array}

\usepackage{tikz}
\usetikzlibrary{
  arrows.meta,
  calc,
  positioning,
  decorations.pathreplacing,
  patterns,
  fit,
  backgrounds,
  intersections,
  shapes.geometric  
}

\hypersetup{
  colorlinks=true,
  linkcolor=blue!60!black,
  citecolor=blue!60!black,
  urlcolor=blue!60!black
}

\newtheorem{theorem}{Theorem}[section]
\newtheorem{proposition}[theorem]{Proposition}
\newtheorem{lemma}[theorem]{Lemma}
\newtheorem{corollary}[theorem]{Corollary}
\theoremstyle{definition}
\newtheorem{definition}[theorem]{Definition}
\newtheorem{remark}[theorem]{Remark}
\newtheorem{example}[theorem]{Example}

\newtheorem{solution*}{Solution}[section]

\newcommand{\R}{\mathbb{R}}
\newcommand{\Fcal}{\mathcal{F}}
\newcommand{\Lcal}{\mathcal{L}}
\newcommand{\Kcal}{\mathcal{K}}

\newcommand{\Ical}{\mathcal{I}}

\newcommand{\Gr}{\operatorname{Gr}}
\newcommand{\SOL}{\operatorname{SOL}}
\newcommand{\dist}{\operatorname{dist}}
\newcommand{\AVI}{\operatorname{AVI}}

\title{First-Order Optimality Conditions for Mathematical Programming with Equilibrium Constraints}
\author{Louis Shuo Wang\thanks{Email: wang.s41@northeastern.edu}}
\affil{Department of Mathematics, Northeastern University, Boston, 02115, MA, USA}

\abstract{
We present a systematic introduction to first-order optimality conditions for mathematical programs with equilibrium constraints (MPECs), emphasizing the limitations of classical nonlinear programming techniques. The goal is twofold. First, we explain why a direct application of standard optimality conditions—based on reformulating MPECs via KKT systems or differentiable exact penalty functions—is often inadequate, as such approaches typically require strong and restrictive assumptions, including nondegeneracy and smoothness conditions.

Second, we develop a first-principles framework for analyzing MPECs by focusing on the geometric structure of the feasible region. In particular, we study stationarity concepts and provide a detailed characterization of the tangent cone at feasible points, which leads to appropriate constraint qualifications tailored to MPECs. These results form the foundation for rigorous first-order analysis and clarify the relationship between the original MPEC formulation and its KKT-based representation, offering practical guidance for handling these inherently challenging optimization problems.}

\begin{document}
\maketitle

\section{Introduction}

A mathematical program with equilibrium constraints (MPEC) is an
optimization problem where the upper-level variables are constrained by the
solution set of a lower-level variational inequality or complementarity
problem. Such models arise in engineering design, network equilibrium,
contact mechanics, economics, and hierarchical control.

A central message of these notes is that a naive KKT reformulation of the
lower-level equilibrium system can obscure the true local geometry of the
MPEC feasible set. In particular, the tangent cone is generally nonconvex,
so the standard nonlinear programming linearization is not the right
first-order object.

\section{Background Prerequisites for Chapter 3}

This section collects the minimum optimization background needed to read
Chapter 3 on first-order optimality conditions for MPECs. The goal is not
to develop full general theory, but to explain the objects that appear in
the chapter and why they matter.

\subsection{Vectors, gradients, Jacobians, and Hessians}

Let $x\in\mathbb{R}^n$ be a vector of variables and let
$\phi:\mathbb{R}^n\to\mathbb{R}$ be differentiable.

\begin{definition}[Gradient]
The gradient of $\phi$ at $x$ is
\[
\nabla \phi(x)
=
\begin{bmatrix}
\frac{\partial \phi}{\partial x_1}(x)\\
\vdots\\
\frac{\partial \phi}{\partial x_n}(x)
\end{bmatrix}.
\]
It gives the first-order change of $\phi$:
\[
\phi(x+d)\approx \phi(x)+\nabla\phi(x)^T d.
\]
\end{definition}

\begin{definition}[Jacobian]
If $g:\mathbb{R}^n\to\mathbb{R}^m$ is differentiable, its Jacobian is the
matrix
\[
\nabla g(x)=
\begin{bmatrix}
\nabla g_1(x)^T\\
\vdots\\
\nabla g_m(x)^T
\end{bmatrix}.
\]
Each row is the gradient of one component function.
\end{definition}

\begin{definition}[Hessian]
If $\phi$ is twice differentiable, its Hessian is
\[
\nabla^2 \phi(x)=
\left[\frac{\partial^2\phi}{\partial x_i\partial x_j}(x)\right]_{i,j=1}^n.
\]
It captures second-order curvature.
\end{definition}

\begin{remark}
In optimization, first-order theory is based mainly on gradients and
Jacobians. Hessians enter second-order theory.
\end{remark}

\subsection{Local minimizers and feasible directions}

Consider the constrained optimization problem
\[
\min_x \ \phi(x)
\qquad \text{subject to } x\in \Omega,
\]
where $\Omega\subseteq \mathbb{R}^n$ is the feasible set.

\begin{definition}[Local minimizer]
A point $\bar x\in \Omega$ is a local minimizer if there exists $\varepsilon>0$
such that
\[
\phi(\bar x)\le \phi(x)
\qquad
\forall x\in \Omega \cap B_\varepsilon(\bar x).
\]
\end{definition}

\begin{definition}[Feasible direction]
A vector $d\in\mathbb{R}^n$ is a feasible direction at $\bar x\in\Omega$ if
one can move from $\bar x$ a little in direction $d$ while staying feasible,
at least to first order.
\end{definition}

The mathematical object that collects all such first-order feasible
directions is the tangent cone.

\subsection{Tangent cone and linearized cone}

\begin{definition}[Tangent cone]
Let $\Omega\subseteq\mathbb{R}^n$ and $\bar x\in\Omega$. The tangent cone to
$\Omega$ at $\bar x$ is
\[
T(\bar x;\Omega)
:=
\left\{
d:
\exists\, x^k\in \Omega,\ x^k\to \bar x,\ \exists\, t_k\downarrow 0
\text{ such that }
\frac{x^k-\bar x}{t_k}\to d
\right\}.
\]
\end{definition}

\begin{remark}
The tangent cone contains all first-order directions in which feasibility can
be approached. It is the correct geometric object for first-order necessary
conditions.
\end{remark}

If $\bar x$ is a local minimizer, then no tangent direction can be a descent
direction for the objective.

\begin{theorem}[Primal first-order necessary condition]
If $\bar x$ is a local minimizer of
\[
\min_{x\in\Omega}\phi(x),
\]
and $\phi$ is differentiable, then
\[
\nabla \phi(\bar x)^T d \ge 0
\qquad \forall d\in T(\bar x;\Omega).
\]
\end{theorem}

In practice, $T(\bar x;\Omega)$ may be hard to compute exactly. For a smooth
nonlinear program
\[
\min_x \ \phi(x)
\quad
\text{s.t.}\quad
g_i(x)\le 0,\ i=1,\dots,m,\qquad
h_j(x)=0,\ j=1,\dots,p,
\]
one often uses the \emph{linearized cone}.

\begin{definition}[Linearized cone]
Let $\bar x$ be feasible and define the active inequality index set
\[
I(\bar x):=\{i:g_i(\bar x)=0\}.
\]
Let $\mathcal{F}\subset \Omega$ be the feasible region for the mathematical program under consideration.
The linearized cone is
\[
\mathcal{L}(\bar x; \mathcal{F})
:=
\left\{
d\in \mathcal{T}(\bar{x};\Omega):
\nabla g_i(\bar x)^T d\le 0 \ \forall i\in I(\bar x),\,\,\,
\nabla h_j(\bar x)^T d=0 \ \forall j
\right\}.
\]
\end{definition}

\begin{remark}
Always,
\[
T(\bar x;\mathcal{F})\subseteq \mathcal{L}(\bar x;\mathcal{F}).
\]
If equality holds, then the abstract tangent-cone condition becomes a
computable linear condition.
\end{remark}

\subsection{Lagrangian and KKT conditions}

For the nonlinear program above, define the Lagrangian
\[
\mathcal{L}(x,\lambda,\mu)
=
\phi(x)+\sum_{i=1}^m \lambda_i g_i(x)+\sum_{j=1}^p \mu_j h_j(x).
\]

\begin{definition}[KKT conditions]
A feasible point $\bar x$ satisfies the KKT conditions if there exist
multipliers $\bar\lambda\in\mathbb{R}^m$ and $\bar\mu\in\mathbb{R}^p$ such that
\begin{align*}
\nabla_x \mathcal{L}(\bar x,\bar\lambda,\bar\mu)&=0
\qquad\text{(stationarity)},\\
g_i(\bar x)&\le 0,\quad i=1,\dots,m
\qquad\text{(primal feasibility)},\\
h_j(\bar x)&=0,\quad j=1,\dots,p
\qquad\text{(primal feasibility)},\\
\bar\lambda_i&\ge 0,\quad i=1,\dots,m
\qquad\text{(dual feasibility)},\\
\bar\lambda_i g_i(\bar x)&=0,\quad i=1,\dots,m
\qquad\text{(complementary slackness)}.
\end{align*}
\end{definition}

\begin{remark}
The KKT system is a \emph{primal--dual} form of first-order optimality.
It augments the primal variable $x$ with multiplier variables.
\end{remark}

\subsection{Constraint qualifications}

KKT conditions are not automatically valid at every local minimizer. Some
regularity assumption is needed.

\begin{definition}[Constraint qualification]
A constraint qualification (CQ) is an assumption ruling out pathological
degeneracy of the constraints, so that tangent-cone geometry and
multiplier-based conditions behave well.
\end{definition}

Two standard examples are LICQ and MFCQ.

\begin{definition}[LICQ]
The linear independence constraint qualification holds at a feasible point
if the gradients of all active inequalities together with the gradients of all
equalities are linearly independent.
\end{definition}

\begin{definition}[MFCQ]
The Mangasarian--Fromovitz constraint qualification holds at a feasible
point if the equality gradients are linearly independent and there exists a
direction that preserves equalities while strictly decreasing all active
inequalities.
\end{definition}

\begin{remark}
Under standard CQs, the tangent cone of an NLP matches the linearized
cone, and KKT conditions hold at local minimizers.
\end{remark}

\subsection{Variational inequalities}

MPECs involve a lower-level equilibrium problem. One standard way to
model equilibrium is through a variational inequality.

\begin{definition}[Variational inequality]\label{def:variation_ineq}
Let $K\subseteq\mathbb{R}^m$ be closed and convex, and let
$F:\mathbb{R}^m\to\mathbb{R}^m$. The variational inequality problem
$\mathrm{VI}(F,K)$ is to find $y\in K$ such that
\[
F(y)^T(v-y)\ge 0
\qquad \forall v\in K.
\]
\end{definition}

\begin{remark}
Interpretation: at the equilibrium point $y$, moving to any other feasible
point $v\in K$ cannot produce a negative first-order directional change in
the equilibrium map.
\end{remark}

\begin{definition}[Solution set]
The solution set of the VI is denoted
\[
\mathrm{SOL}(F,K).
\]
\end{definition}

\begin{example}
If $K=\mathbb{R}^m$, then the VI condition becomes
\[
F(y)=0.
\]
So VIs generalize systems of equations.
\end{example}

\begin{definition}[Affine variational inequality]
If
\[
F(y)=My+q
\]
with matrix $M$ and vector $q$, the problem is called an affine
variational inequality (AVI). We use the notations $\mathrm{AVI}(q,M,K)$ and $\mathrm{SOL}(q,M,K)$ to represent the affine variation problem and its solution set. Here, $K$ is the same as in Definition~\ref{def:variation_ineq}.
\end{definition}

\subsection{Complementarity problems}

A very important special case of a VI is a complementarity system.

\begin{definition}[Nonlinear complementarity problem]
Given $F:\mathbb{R}^m\to\mathbb{R}^m$, the nonlinear complementarity
problem (NCP) is to find $y\in\mathbb{R}^m$ such that
\[
y\ge 0,\qquad F(y)\ge 0,\qquad y_iF_i(y)=0\ \ \forall i.
\]
\end{definition}

\begin{remark}
For each index $i$, complementarity means:
\[
\text{either } y_i=0 \text{ or } F_i(y)=0,
\]
possibly both. Thus the product condition enforces a switch-like behavior.
\end{remark}

\begin{definition}[Linear complementarity problem]
If $F(y)=My+q$, then the NCP becomes a linear complementarity problem
(LCP):
\[
y\ge 0,\qquad My+q\ge 0,\qquad y^T(My+q)=0.
\]
\end{definition}

\begin{remark}
Complementarity systems appear naturally in contact mechanics,
equilibrium networks, market clearing, and KKT systems of optimization
problems.
\end{remark}

\subsection{From lower-level equilibrium to MPEC}

An MPEC combines an upper-level optimization problem with a lower-level
equilibrium problem.

\begin{definition}[MPEC: informal form]
A mathematical program with equilibrium constraints has the form
\[
\min_{x,y} f(x,y)
\qquad
\text{s.t. } (x,y)\in Z,\quad y\in S(x),
\]
where $S(x)$ is the solution set of a lower-level VI or complementarity
problem.
\end{definition}

\begin{remark}
The variable $x$ is an upper-level design/control variable, while $y$
must satisfy an equilibrium condition depending on $x$.
\end{remark}

At first sight one may think: ``replace the lower-level VI by its KKT
conditions and then apply ordinary nonlinear programming theory.'' The
main lesson of Chapter 3 is that this is generally too naive.

\subsection{Why MPEC is harder than ordinary NLP}

For ordinary NLPs, under a suitable CQ, the tangent cone is often a single
polyhedral cone described by linearized constraints. For MPECs, however,
the feasible set may have a kinked or branch-like local structure.

\begin{remark}
Near a feasible point, an MPEC feasible set may look like the union of
several smooth pieces. Therefore its tangent cone can be a \emph{union of
polyhedral cones}, and that union is generally nonconvex.
\end{remark}

This is exactly why Chapter 3 introduces an MPEC-specific linearized cone
instead of relying on the single-cone NLP picture.

\subsection{Polyhedral cones}

\begin{definition}[Cone]
A set $K\subseteq\mathbb{R}^n$ is a cone if for every $d\in K$ and every
$\alpha\ge 0$, one has $\alpha d\in K$.
\end{definition}

\begin{definition}[Polyhedral cone]
A cone is polyhedral if it can be written as the solution set of finitely many
linear equalities and inequalities.
\end{definition}

\begin{example}
The nonnegative orthant
\[
\mathbb{R}^n_+ = \{x\in\mathbb{R}^n:x_i\ge 0,\ i=1,\dots,n\}
\]
is a polyhedral cone.
\end{example}

\begin{remark}
Polyhedral cones are important because optimization over them is much more
tractable than over general nonlinear sets.
\end{remark}

\subsection{Primal versus primal--dual stationarity}

\begin{definition}[Primal stationarity]
A first-order condition written directly in terms of feasible directions,
for example
\[
\nabla f(\bar x)^T d\ge 0
\qquad \forall d\in T(\bar x;\Omega),
\]
is called a primal stationarity condition.
\end{definition}

\begin{definition}[Primal--dual stationarity]
A first-order condition written using both primal variables and
multiplier variables is called a primal--dual stationarity condition.
KKT conditions are the standard example.
\end{definition}

\begin{remark}
Chapter 3 begins with primal stationarity for MPEC and then, under
appropriate MPEC-specific constraint qualifications, derives a more
computable primal--dual form.
\end{remark}

\subsection{The role of the chapter-specific objects}

The following terms in Chapter 3 are specializations of the background
concepts above.

\begin{definition}[Lower-level multiplier set]
The set $M(z)$ collects multipliers associated with the KKT system of the
lower-level variational inequality at a feasible point $z=(x,y)$.
\end{definition}

\begin{definition}[Lifted critical cone]
The lifted critical cone $K(z,\lambda)$ is an MPEC-specific cone built from
active lower-level constraints and a lower-level multiplier $\lambda$. It
refines the usual linearized-cone idea.
\end{definition}

\begin{definition}[Directional critical set]
For a fixed upper-level direction $dx$, the set $K(z,\lambda;dx)$ collects
the lower-level directions $dy$ compatible with that $dx$.
\end{definition}

\begin{definition}[Affine variational inequality arising from linearization]
After linearizing the lower-level VI at a feasible point, one obtains an AVI
in the variable $dy$. Solving that AVI describes candidate tangent
directions of the MPEC.
\end{definition}

\begin{remark}
So the chapter replaces the single linearized cone of ordinary NLP by a
family of cones and AVIs indexed by lower-level multipliers.
\end{remark}

\section{The MPEC Model and Stationarity}

\begin{definition}[MPEC]
Let $f:\R^{n+m}\to\R$, let $Z\subseteq \R^{n+m}$, and define
\[
C(x):=\{\,y\in\R^m:\ g(x,y)\le 0\,\}.
\]
Let
\[
S(x):=\SOL(F(x,\cdot),C(x)).
\]
The mathematical program with equilibrium constraints is
\begin{equation}
\label{eq:MPEC}
\begin{aligned}
\min_{x,y}\quad & f(x,y)\\
\text{s.t.}\quad & (x,y)\in Z,\\
& y\in S(x).
\end{aligned}
\end{equation}
Its feasible set is denoted by $\Fcal$.
\end{definition}

\begin{definition}[Multiplier set]
For a feasible point $z=(x,y)\in \Fcal$, define the active index set
\[
\Ical(z):=\{\,i:\ g_i(z)=0\,\}.
\]
The multiplier set of the lower-level variational inequality is
\[
M(z):=
\left\{
\lambda\in\R^\ell:
F(z)+\sum_{i=1}^{\ell}\lambda_i\nabla_y g_i(z)=0,\ 
\lambda\ge 0,\ 
g(z)\le 0,\ 
\lambda^T g(z)=0
\right\}.
\]
\end{definition}

\begin{definition}[Stationary point]
A feasible point $\bar z=(\bar x,\bar y)\in\Fcal$ is called a
\emph{stationary point} of \eqref{eq:MPEC} if
\[
\nabla_x f(\bar z)^Tdx+\nabla_y f(\bar z)^Tdy\ge 0
\qquad \forall (dx,dy)\in T(\bar z;\Fcal),
\]
where $T(\bar z;\Fcal)$ is the tangent cone to the feasible set.
\end{definition}

\begin{remark}
Every local minimizer of the MPEC is stationary in the above sense.
\end{remark}

We give an example to show why the standard NLP viewpoint fails.
\begin{example}[Nonconvex tangent cone]
Consider an MPEC whose feasible region near $\bar z$ is the union of two
smooth branches. Then
\[
T(\bar z;\Fcal)=K_1\cup K_2
\]
for two polyhedral cones $K_1,K_2$, so the tangent cone is nonconvex.
This already shows that a single convex NLP-type linearized cone cannot
capture the correct first-order geometry.
\end{example}

\subsection{The Tangent Cone and the MPEC Linearization}

Fix a feasible point $\bar z=(\bar x,\bar y)\in\Fcal$.

\begin{definition}[SBCQ]
The \emph{sequence boundedness constraint qualification} (SBCQ) holds at
$\bar z$ if for every sequence $\{(x^k,y^k)\}\subset \Fcal$ converging to
$\bar z$, there exists for each $k$ a multiplier $\lambda^k\in M(x^k,y^k)$
such that the sequence $\{\lambda^k\}$ is bounded.
\end{definition}

\begin{definition}[VI Lagrangian]
The Lagrangian of the lower-level variational inequality is
\[
L(z,\lambda):=F(z)+\sum_{i=1}^{\ell}\lambda_i\nabla_y g_i(z),
\qquad (z,\lambda)\in \R^{n+m+\ell}.
\]
\end{definition}

\begin{definition}[Lifted critical cone]
For $\lambda\in M(\bar z)$, define the \emph{lifted critical cone}
\[
\Kcal(\bar z,\lambda):=
\left\{
(dx,dy)\in\R^{n+m}:
\begin{array}{l}
dx^T\nabla_x g_i(\bar z)+dy^T\nabla_y g_i(\bar z)\le 0,
\quad i\in\Ical(\bar z),\ \lambda_i=0,\\[0.3em]
dx^T\nabla_x g_i(\bar z)+dy^T\nabla_y g_i(\bar z)=0,
\quad i\in\Ical(\bar z),\ \lambda_i>0
\end{array}
\right\}.
\]
\end{definition}

\begin{definition}[Directional critical set]
For $dx\in\R^n$ and $\lambda\in M(\bar z)$, define
\[
\Kcal(\bar z,\lambda;dx):=
\{\,dy\in\R^m:(dx,dy)\in \Kcal(\bar z,\lambda)\,\}.
\]
\end{definition}

\begin{definition}[Set-valued linearization map]
For each $\lambda\in M(\bar z)$, define
\[
\mathscr{L}\mathscr{S}_{(\bar z,\lambda)}(dx):=
\SOL\bigl(\nabla_xL(\bar z,\lambda)\,dx,\,
          \nabla_yL(\bar z,\lambda),\,
          \Kcal(\bar z,\lambda;dx)\bigr).
\]
\end{definition}

\begin{definition}[MPEC linearized cone]
The \emph{MPEC linearized cone} at $\bar z$ is
\[
\Lcal(\bar z;\Fcal):=
T(\bar z;Z)\cap
\left(
\bigcup_{\lambda\in M(\bar z)}
\Gr(\mathscr{L}\mathscr{S}_{(\bar z,\lambda)})
\right).
\]
\end{definition}

\begin{remark}
The cone $\Lcal(\bar z;\Fcal)$ is generally nonconvex. In many cases it is a
union of finitely many polyhedral cones.
\end{remark}

\begin{lemma}[Tangent directions induce an affine VI]
\label{lem:321}
Assume SBCQ holds at $\bar z\in\Fcal$. If
\[
dz=(dx,dy)\in T(\bar z;\Fcal),
\]
then $dz\in T(\bar z;Z)$ and there exists a multiplier
$\bar\lambda\in M(\bar z)$ such that $dy$ solves the AVI
\[
\AVI\bigl(\nabla_xL(\bar z,\bar\lambda)dx,\,
          \nabla_yL(\bar z,\bar\lambda),\,
          \Kcal(\bar z,\bar\lambda;dx)\bigr).
\]
Equivalently,
\[
T(\bar z;\Fcal)\subseteq \Lcal(\bar z;\Fcal).
\]
\end{lemma}

\begin{proof}[Proof sketch]
Take a tangent sequence $z^k=(x^k,y^k)\in\Fcal$ converging to $\bar z$
with difference quotients converging to $(dx,dy)$. By SBCQ one may
choose bounded multipliers $\lambda^k\in M(z^k)$; after extracting a
subsequence, $\lambda^k\to\bar\lambda\in M(\bar z)$. First-order expansion
of the active constraints yields $(dx,dy)\in\Kcal(\bar z,\bar\lambda)$, and
first-order expansion of the VI equation yields the stated AVI.
\end{proof}

\begin{definition}[Extreme-point cone]
Let $M^e(\bar z)$ denote the set of extreme points of $M(\bar z)$. Define
\[
\Lcal^e(\bar z;\Fcal):=
T(\bar z;Z)\cap
\left(
\bigcup_{\lambda\in M^e(\bar z)}
\Gr(\mathscr{L}\mathscr{S}_{(\bar z,\lambda)})
\right).
\]
\end{definition}

\begin{proposition}
\label{prop:322}
If CRCQ holds at $\bar z\in\Fcal$, then
\[
T(\bar z;\Fcal)\subseteq \Lcal^e(\bar z;\Fcal).
\]
If, in addition, $T(\bar z;Z)$ is polyhedral, then
$\Lcal^e(\bar z;\Fcal)$ is the union of finitely many polyhedral cones.
\end{proposition}

\subsection{Critical Multipliers and Directional Structure}

\begin{definition}[Critical multipliers]
For a given direction $dx\in\R^n$, define $M^c(\bar z;dx)$ as the set of
optimal solutions of the linear program
\[
\max_{\lambda\in M(\bar z)}
\sum_{i=1}^{\ell}\lambda_i\,dx^T\nabla_x g_i(\bar z).
\]
\end{definition}

\begin{proposition}
\label{prop:323}
Let $\lambda\in M(\bar z)$. Then
\[
\Kcal(\bar z,\lambda;dx)\neq \varnothing
\quad\Longleftrightarrow\quad
\lambda\in M^c(\bar z;dx).
\]
Moreover, for every $\lambda\in M^c(\bar z;dx)$, the set
$\Kcal(\bar z,\lambda;dx)$ is independent of $\lambda$ and coincides with
the optimal solution set of the dual linear program
\[
\min_{dy\in\R^m}\ dy^T F(\bar z)
\quad\text{s.t.}\quad
dx^T\nabla_x g_i(\bar z)+dy^T\nabla_y g_i(\bar z)\le 0,
\quad i\in \Ical(\bar z).
\]
\end{proposition}

\begin{proof}[Idea]
If $dy\in\Kcal(\bar z,\lambda;dx)$, then the complementarity relations give
the slackness conditions needed for optimality of $\lambda$ in the primal
LP. Conversely, if $\lambda$ is primal-optimal, then an optimal dual
solution $dy$ exists and complementary slackness implies
$dy\in \Kcal(\bar z,\lambda;dx)$.
\end{proof}

\subsection{Constraint Qualifications for MPEC}

\begin{definition}[Basic, extreme, and full CQ]
We say that the \emph{basic constraint qualification} holds at
$\bar z\in\Fcal$ if there exists a nonempty subset $M'\subseteq M(\bar z)$
such that
\[
T(\bar z;\Fcal)=\Lcal'(\bar z;\Fcal),
\]
where
\[
\Lcal'(\bar z;\Fcal):=
T(\bar z;Z)\cap
\left(
\bigcup_{\lambda\in M'}
\Gr(\mathscr{L}\mathscr{S}_{(\bar z,\lambda)})
\right).
\]
If this holds with $M'=M^e(\bar z)$, we call it the \emph{extreme CQ}. If
it holds with $M'=M(\bar z)$, we call it the \emph{full CQ}.
\end{definition}

\begin{remark}
If $M(\bar z)$ is a singleton, then the basic, extreme, and full CQs
coincide.
\end{remark}

\subsection{Stationarity Under the Full CQ}

\begin{corollary}[Equivalent primal characterizations of stationarity]
\label{cor:331}
Assume SBCQ holds at $\bar z\in\Fcal$ and the full CQ holds at $\bar z$,
that is,
\[
\Lcal(\bar z;\Fcal)=T(\bar z;\Fcal).
\]
Then the following are equivalent:
\begin{enumerate}[label=\textnormal{(\alph*)}]
\item $\bar z$ is a stationary point of the MPEC;
\item
\[
dz\in \Lcal(\bar z;\Fcal)\Longrightarrow \nabla f(\bar z)^Tdz\ge 0;
\]
\item $dz=0$ is an optimal solution of
\[
\min_{dz}\ \nabla f(\bar z)^Tdz
\qquad\text{s.t.}\qquad dz\in \Lcal(\bar z;\Fcal);
\]
\item for every scalar $\alpha\ge \|\nabla f(\bar z)\|$ and every
$dz\in\R^{n+m}$,
\[
\nabla f(\bar z)^Tdz+\alpha\,\dist(dz,\Lcal(\bar z;\Fcal))\ge 0.
\]
\end{enumerate}
\end{corollary}

\begin{proof}
The equivalence of (a) and (b) follows directly from the definition of
stationarity and the equality $T(\bar z;\Fcal)=\Lcal(\bar z;\Fcal)$. The
equivalence of (b) and (c) is immediate. The implication (d)$\Rightarrow$(c)
is obvious. For (c)$\Rightarrow$(d), choose $dz'\in\Lcal(\bar z;\Fcal)$ with
\[
\|dz-dz'\|=\dist(dz,\Lcal(\bar z;\Fcal)).
\]
Then
\[
\nabla f(\bar z)^Tdz+\alpha\,\dist(dz,\Lcal(\bar z;\Fcal))
\ge
\nabla f(\bar z)^Tdz'+(\alpha-\|\nabla f(\bar z)\|)\|dz-dz'\|\ge 0.
\qedhere
\]
\end{proof}

\subsubsection{Primal--dual stationarity}

Assume now that the upper-level feasible set is polyhedral:
\[
Z=\{(x,y)\in\R^{n+m}: Gx+Hy+a\le 0\},
\]
where $G\in\R^{s\times n}$, $H\in\R^{s\times m}$, and $a\in\R^s$.

\begin{definition}[Degenerate and positive multiplier index sets]
For $\lambda\in M(\bar z)$, define
\[
I_0(\bar z,\lambda):=\{\,i\in\Ical(\bar z):\lambda_i=0\,\},
\qquad
I_+(\bar z,\lambda):=\{\,i\in\Ical(\bar z):\lambda_i>0\,\}.
\]
\end{definition}

\begin{theorem}[Primal--dual characterization under the full CQ]
\label{thm:334}
Under the assumptions of Corollary~\ref{cor:331}, $\bar z$ is a stationary
point of \eqref{eq:MPEC} if and only if for each vector
$\lambda\in M(\bar z)$ and each pair of index sets
$\alpha,\bar\alpha$ partitioning $I_0(\bar z,\lambda)$, there exist
multipliers
\[
\zeta\in\R^s,\qquad \eta\in\R^\ell,\qquad \pi\in\R^m
\]
such that
\begin{align*}
\nabla_x f(\bar z)+G^T\zeta+\nabla_x g(\bar z)^T\eta
&=
\nabla_x L(\bar z,\lambda)^T\pi,
\\
\nabla_y f(\bar z)+H^T\zeta+\nabla_y g(\bar z)^T\eta
&=
\nabla_y L(\bar z,\lambda)^T\pi,
\\
\pi^T\nabla_y g_i(\bar z) &\le 0,
\qquad i\in\alpha,
\\
\pi^T\nabla_y g_i(\bar z) &= 0,
\qquad i\in I_+(\bar z,\lambda),
\\
\eta_i &\ge 0,
\qquad i\in\bar\alpha,
\\
\eta_i &= 0,
\qquad i\notin \Ical(\bar z),
\\
\zeta &\ge 0,
\qquad
\zeta^T(G\bar x+H\bar y+a)=0.
\end{align*}
\end{theorem}

\begin{corollary}[Finite verification under CRCQ and extreme CQ]
If, instead of SBCQ and full CQ, one assumes CRCQ and extreme CQ at
$\bar z$, then Theorem~\ref{thm:334} remains valid with $M(\bar z)$
replaced by the finite set $M^e(\bar z)$. Hence stationarity reduces to
finitely many linear inequality systems.
\end{corollary}

\subsection{Special Case: NCP-Constrained MPECs}

\begin{definition}[NCP-constrained MPEC]
An important special case occurs when
\[
g(x,y)=-y.
\]
Then the lower-level equilibrium condition becomes the nonlinear
complementarity system
\[
y\ge 0,\qquad F(x,y)\ge 0,\qquad y^T F(x,y)=0.
\]
\end{definition}

\begin{definition}[NCP index sets]
For a feasible point $\bar z=(\bar x,\bar y)$, define
\begin{align*}
\alpha(\bar z)&:=\{\,i:\ \bar y_i>0=F_i(\bar z)\,\},\\
\beta(\bar z)&:=\{\,i:\ \bar y_i=0=F_i(\bar z)\,\},\\
\gamma(\bar z)&:=\{\,i:\ \bar y_i=0<F_i(\bar z)\,\}.
\end{align*}
The indices in $\beta(\bar z)$ are called \emph{degenerate}.
\end{definition}

\begin{remark}
In the NCP case, the multiplier set is a singleton:
\[
M(\bar z)=\{F(\bar z)\}.
\]
Moreover, for any $dx\in\R^n$, the directional critical set becomes the
constant cone
\[
\R^{|\alpha(\bar z)|}\times \R_+^{|\beta(\bar z)|}\times \{0\}^{|\gamma(\bar z)|},
\]
and the affine variational inequality in Lemma~\ref{lem:321} reduces to a
mixed linear complementarity problem.
\end{remark}

\subsection{Engineering Interpretation}

\begin{remark}[Why this matters in engineering]
In engineering applications, an equilibrium subproblem often models
contact, traffic, flow balance, or market clearing, while the upper-level
optimization controls geometry, design parameters, or operating points.
The results above show that:
\begin{enumerate}[label=\arabic*.]
\item local optimality must be checked against the true tangent cone,
      not merely an NLP reformulation;
\item the MPEC linearized cone is the correct first-order replacement;
\item suitable MPEC-specific constraint qualifications are essential;
\item stationarity can be expressed in both primal and primal--dual form.
\end{enumerate}
\end{remark}

\section{Exercise}

\subsection{True or False}

\begin{enumerate}[label=(\roman*)]
\item \textbf{MPEC Tangent Cone Geometry}:
    
In standard nonlinear programming, the basic constraint qualification ensures the tangent cone is polyhedral. In contrast, the tangent cone of an MPEC’s feasible region at a feasible point is generally the union of several polyhedral convex cones.

\item \textbf{Extreme vs. Full Constraint Qualifications (CQ)}:

According to the text, the full constraint qualification (CQ) for an MPEC logically implies the extreme CQ without requiring any additional mathematical assumptions.

\item \textbf{Strong Nondegeneracy}:

A feasible solution $\overline{z}$ of the MPEC is considered "strongly nondegenerate" if and only if the set of active gradients is linearly independent, which consequently implies the multiplier set $M(\overline{z})$ is empty.

\item \textbf{Local Minima Equivalence}:

If there exists at least one multiplier $\overline{\lambda} \in M(\overline{z})$ such that the triple $(\overline{x},\overline{y},\overline{\lambda})$ is a local minimizer of the corresponding KKT formulation, then the pair $(\overline{x},\overline{y})$ is guaranteed to be a local minimizer of the original MPEC.

\item \textbf{Primal Stationarity under SBCQ}:

Assuming the Strictly Binding Constraint Qualification (SBCQ) and the full CQ hold at $\overline{z}$, $\overline{z}$ is a stationary point of the MPEC if and only if $\nabla f(\overline{z}) \in \mathcal{L}(\overline{z};\mathcal{F})^*$.

\item \textbf{The Strict Mangasarian-Fromovitz CQ (SMFCQ)}:

The Strict Mangasarian-Fromovitz Constraint Qualification (SMFCQ) holds at $(\overline{x},\overline{y},\overline{\lambda})$ if and only if the multiplier set $M(\overline{z})$ is a singleton containing exactly $\overline{\lambda}$.

\item \textbf{The Directional Critical Set}:

For any given direction $dx \in \mathbb{R}^n$ and multiplier $\lambda \in M(\overline{z})$, the directional critical set $\mathcal{K}(\overline{z},\lambda;dx)$ is non-empty if and only if $\lambda$ belongs to the set of critical multipliers $M^c(\overline{z};dx)$.

\item \textbf{Complexity of Primal-Dual Stationarity}:

When deriving the primal-dual stationarity conditions for an MPEC where the multiplier set $M(\overline{z})$ is a singleton (e.g., $M(\overline{z}) = \{\overline{\lambda}\}$), the complementarity condition decomposes into exactly one system of equalities and inequalities, avoiding the disjunctive complexity seen in non-singleton sets.

\item \textbf{Polyhedral Representation of the Linearized Cone}:

If the Constant Rank Constraint Qualification (CRCQ) holds at $\overline{z}$ and the tangent cone $\mathcal{T}(\overline{z};Z)$ is polyhedral, the specialized cone $\mathcal{L}^e(\overline{z};\mathcal{F})$ is guaranteed to be the union of a finite number of polyhedral cones.

\item \textbf{CQs and the KKT Formulation Tangent Cone}:

If the SMFCQ and the CRCQ both fail at a feasible point $\overline{z}$, it is mathematically impossible to prove that the tangent cone $\mathcal{T}(\overline{z};\mathcal{F})$ contains the projection of $\mathcal{T}((\overline{z},\lambda);\mathcal{F}^{KKT})$ onto $\mathbb{R}^{n+m}$.
\end{enumerate}

\subsection{Computational questions}
\paragraph{Question 1. Basic CQ Failure.}
Consider the MPEC formulation where $F(x,y) = y$ and $g(x,y) = -y$ for $(x,y) \in \mathbb{R}^2$. Recast this using the standard KKT formulation to obtain the NLP: $\min f(x,y)$ s.t. $y-\lambda=0$, $y\lambda=0$, $y,\lambda \ge 0$. 
    For an arbitrary feasible point $\overline{z} = (x,0,0)$, explicitly compute the tangent cone $\mathcal{T}(\overline{z}; \mathcal{F}^{NLP})$ and the linearized cone $\mathcal{L}(\overline{z}; \mathcal{F}^{NLP})$. Use your result to prove that the basic NLP constraint qualification fails.

    \paragraph{Question 2. Direct Tangent Cone Computation.} Consider an MPEC in $\mathbb{R}^3$ with $F(x,y_1,y_2) = (x^2+x+y_1, x^3+y_2)^T$ and $g(x,y_1,y_2) = -(y_1,y_2)^T$. The inner VI is cast as a complementarity system: $x^2+x+y_1 \ge 0 \perp y_1 \ge 0$ and $x^3+y_2 \ge 0 \perp y_2 \ge 0$. 
    At the global optimum $\overline{z} = (0,0,0)$, compute the exact set of limit vectors $dz = \lim_{k\to\infty} \tau_k^{-1}(z^k - \overline{z})$ to explicitly derive the non-convex tangent cone $\mathcal{T}(\overline{z}; \mathcal{F})$.

    \paragraph{Question 3. Linearized Cone and Convex Hulls.} Using the same MPEC from Question 2 at $\overline{z} = (0,0,0)$, compute the linearized cone $\mathcal{L}(\overline{z}; \mathcal{F}^{NLP})$ based on the standard nonlinear programming formulation. Prove via direct algebraic comparison that the convex hull of $\mathcal{T}(\overline{z}; \mathcal{F})$ is a proper subset of $\mathcal{L}(\overline{z}; \mathcal{F}^{NLP})$.

    \paragraph{Question 4. Multiplier Set and Extreme Points.} Let the inner VI of an MPEC be defined by $\min_{y} \frac{1}{2}\|y-x\|^2$ subject to $C = \{y \in \mathbb{R}^2 : y^T y \le 1, y_1 \le 1\}$. 
    Given $\overline{x} = (2,0)$ and $\overline{y} = (1,0)$, compute the complete Lagrange multiplier set $M(\overline{z})$ for the lower-level problem. Identify the extreme points $M^e(\overline{z})$ and determine whether the Strict Mangasarian-Fromovitz Constraint Qualification (SMFCQ) holds.

    \paragraph{Question 5. Directional Critical Sets.} Continuing with the system from Question 4, select the extreme multiplier $\overline{\lambda} = (1/2, 0) \in M^e(\overline{z})$. For an arbitrary direction $dx = (dx_1, dx_2)^T$, compute the polyhedral directional critical set $\mathcal{K}(\overline{z}, \overline{\lambda}; dx)$ by setting up and solving the corresponding inequalities.

    \paragraph{Question 6. Full CQ Evaluation in NCP Constraints.} 
    Consider an MPEC with 1 upper-level and 1 lower-level variable given by $Z \equiv \mathbb{R}_+ \times \mathbb{R}$, $F(x,y) \equiv y^2+x^2$, and $C(x) \equiv \mathbb{R}_+$.
    At the feasible point $\overline{z} = (0,0)$, compute both $\mathcal{T}(\overline{z}; \mathcal{F})$ and $\mathcal{L}(\overline{z}; \mathcal{F})$. Does the full CQ hold at this point? Justify your computation.

    \paragraph{Question 7. Primal-Dual NCP Stationarity.} Consider the MPEC: $\min x^2+x+2y_1+y_2^2$ subject to $0 \le y_1 \perp x^2+x+y_1 \ge 0$ and $0 \le y_2 \perp x^3+y_2 \ge 0$. 
    At the solution $\overline{z}=(0,0,0)$, identify the degenerate index set $\beta(\overline{z})$. Then, explicitly write out the stationarity system (from Theorem 3.3.6) for the specific subset $\beta_1 = \{1\}$ and compute a valid set of multipliers $(\pi_1, \pi_2)$ to satisfy the system.

    \paragraph{Question 8. Dual Program for Critical Sets.} Recall that the directional critical set $\mathcal{K}(\overline{z}, \lambda; dx)$ can be characterized via a dual linear program: $\min dy^T F(\overline{z})$ subject to $dx^T \nabla_x g_i(\overline{z}) + dy^T \nabla_y g_i(\overline{z}) \le 0$ for $i \in \mathcal{I}(\overline{z})$.
    Formulate and solve this linear program manually for the problem in Question 4 at $\overline{z} = ((2,0), (1,0))$, leaving your answer in terms of an arbitrary $dx \in \mathbb{R}^2$.

    \paragraph{Question 9. Strong Nondegeneracy Verification.} Proposition 3.3.8 states that a point $\overline{z}$ is strongly nondegenerate if and only if the active gradients $\{\nabla_y g_i(\overline{z}) : i \in \mathcal{I}(\overline{z})\}$ are linearly independent and strictly complementary ($\overline{\lambda} - g(\overline{x},\overline{y}) > 0$).
    Consider $F(x,y) = (1,1)^T$ and $g(x,y) = (y_1-x_1, y_2-x_2^2)^T \le 0$. At $\overline{z} = (0,0,0,0)$, evaluate $M(\overline{z})$ and compute the active gradients to prove or disprove strong nondegeneracy.

    \paragraph{Question 10. Local Minima and KKT Equivalence.} Proposition 3.4.1 asserts that $\overline{z}$ is a local minimizer of an MPEC if and only if $(\overline{x},\overline{y},\overline{\lambda})$ is a local minimizer of its KKT formulation for \textit{all} $\overline{\lambda} \in M(\overline{z})$. 
    Return to the setup of Question 4 with the objective $f(x,y) = y_1$. For the non-extreme multiplier $\lambda^* = (0,1) \in M(\overline{z})$, prove by computing the objective values in a local neighborhood $W \times \mathbb{R}^l$ that $(\overline{x}, \overline{y}, \lambda^*)$ is not a local minimizer, thereby showing that the ``for all'' requirement cannot be relaxed.

\subsection{Proof questions}

\paragraph{Question 1. Topological Structure of the Linearized Cone under CRCQ} 
    Let $\overline{z} \in \mathcal{F}$ be a feasible point of an MPEC. The text defines the linearized cone $\mathcal{L}(\overline{z}; \mathcal{F})$ using the entire multiplier set $M(\overline{z})$, and $\mathcal{L}^e(\overline{z}; \mathcal{F})$ using only the extreme points $M^e(\overline{z})$. 
    \begin{enumerate}
        \item Prove rigorously that if the Constant Rank Constraint Qualification (CRCQ) holds at $\overline{z}$, then $\mathcal{T}(\overline{z}; \mathcal{F}) \subseteq \mathcal{L}^e(\overline{z}; \mathcal{F})$. 
        \item Detail the exact step in the limiting sequence argument where the CRCQ is strictly necessary to transition from an arbitrary sequence of multipliers $\{\lambda^k\} \subset M(z^k)$ to a limiting extreme point $\overline{\lambda} \in M^e(\overline{z})$.
    \end{enumerate}

    \paragraph{Question 2. Pathology of Constraint Qualifications}
    
    The text states that $\text{Full CQ} \Rightarrow \text{Extreme CQ} \Rightarrow \text{Basic CQ}$ (with the first implication requiring CRCQ). 
    In Example 3.4.2, the basic CQ holds, but both the extreme and full CQs fail. 
    Analyze the geometric geometry of the mapping $\mathcal{LS}_{(\overline{x},\lambda)}$ in this example. Formally prove why the union of the graphs of the solution sets over the extreme multipliers $M^e(\overline{z})$ is a proper, non-dense subset of $\mathcal{T}(\overline{z}; \mathcal{F})$, forcing the failure of the Extreme CQ.

    \paragraph{Question 3. Derivation of Disjunctive Stationarity Conditions}
    Theorem 3.3.4 provides the primal-dual stationarity conditions for a general MPEC by decomposing the index set $\mathcal{I}_0(\overline{z},\lambda)$. 
    Consider the special case of an NCP-constrained Mathematical Program (Theorem 3.3.6). Rigorously derive the system of equations and inequalities in Theorem 3.3.6 starting directly from the general conditions in Theorem 3.3.4. Explicitly show how the partition of $\mathcal{I}_0(\overline{z},\lambda)$ maps bijectively to the power set of the degenerate index set $\beta(\overline{z})$.

    \paragraph{Question 4. SMFCQ vs. MFCQ in Tangent Cone Projections.}
    Proposition 3.4.3(a) demonstrates that $\mathcal{T}(\overline{z}; \mathcal{F})$ is the projection of the KKT tangent cone $\mathcal{T}((\overline{z},\lambda); \mathcal{F}^{KKT})$ onto $\mathbb{R}^{n+m}$, provided the Strict Mangasarian-Fromovitz CQ (SMFCQ) holds.
    \begin{enumerate}
        \item Assume only the standard Mangasarian-Fromovitz CQ (MFCQ) holds (meaning $M(\overline{z})$ is compact but not necessarily a singleton). Construct a mathematically rigorous counter-example showing that the projection of the KKT tangent cone can be a strict subset of $\mathcal{T}(\overline{z}; \mathcal{F})$.
        \item Explain the failure geometrically in terms of the unboundedness of the normalized direction sequence $\overline{d}$ used in the proof of Proposition 3.4.3.
    \end{enumerate}

    \paragraph{Question 5. Strong Nondegeneracy and Manifold Collapse.}
    If a feasible solution $\overline{z}$ is strongly nondegenerate, Proposition 3.3.9 shows that the stationarity conditions collapse from a disjunctive set of linear systems to a single linear system.
    Prove that under strong nondegeneracy, the MPEC tangent cone $\mathcal{T}(\overline{z}; \mathcal{F})$ is locally isomorphic to the tangent space of a standard smooth nonlinear programming manifold. Conclude your proof by demonstrating why the multiplier for the complementarity constraint inherently vanishes from the MPEC Lagrangian formulation in this specific regime.

    \paragraph{Question 6. Equivalence of Local Minima.}
    Proposition 3.4.1 establishes that $\overline{z}$ is a local minimizer of the MPEC if and only if $(\overline{x},\overline{y},\overline{\lambda})$ is a local minimizer of the equivalent KKT formulation for \textit{every} $\overline{\lambda} \in M(\overline{z})$. 
    Suppose you replace the standard VI constraint with a generalized variational inequality (GVI) where the base set $C(x)$ is non-convex but satisfies a local polyhedral property. Does the ``for all'' requirement in Proposition 3.4.1 still hold as a necessary and sufficient condition? Prove your assertion or provide a counter-example.

\section{Solutions}
\subsection{Solutions for True or False}

\begin{enumerate}[label=(\roman*)]
\item True.
The text explicitly notes that unlike traditional NLP where CQs ensure a polyhedral tangent cone, the tangent cone $\mathcal{T}(\overline{z};\mathcal{F})$ of an MPEC is generally nonconvex and formed by the union of polyhedral convex cones.

\item False.
The text specifies that the implication ``full CQ $\Rightarrow$ extreme CQ'' is valid provided that the Constant Rank Constraint Qualification (CRCQ) holds. It is not an unconditional implication.

\item False.
Proposition 3.3.8 states that linear independence of the active gradients implies the multiplier set $M(\overline{z})$ is a singleton, not empty. Furthermore, strict complementarity must also hold ($\overline{\lambda} - g(\overline{x},\overline{y}) > 0$).

\item False.
While this "there exists" condition holds for global minima, Proposition 3.4.1 and Example 3.4.2 demonstrate that for local minima, the requirement cannot be relaxed. The KKT triple must be a local minimizer for every $\overline{\lambda} \in M(\overline{z})$ to guarantee the pair is a local minimizer of the MPEC.

\item True.
This is a direct restatement of Corollary 3.3.1 (statements a and b), which provides the primal characterization of stationarity for the MPEC.

\item True.
The text explicitly notes in Section 3.2 that it is not difficult to show that the SMFCQ holds at a given triple if and only if $M(\overline{z}) = \{\overline{\lambda}\}$.

\item True. 
This is directly affirmed by Proposition 3.2.3, which establishes the necessary and sufficient condition for the directional critical set to be non-empty based on the dual linear program.

\item False.
The text clarifies that even if $M(\overline{z})$ is a singleton, the complementarity condition still decomposes into $2^c$ systems of equalities and inequalities, where $c$ is the cardinality of the degenerate index set $\mathcal{I}_0(\overline{z},\overline{\lambda})$.

\item True.  
Proposition 3.2.2 guarantees this. Because $M^e(\overline{z})$ is a finite set of extreme points, the resulting cone $\mathcal{L}^e(\overline{z};\mathcal{F})$ simplifies to a finite union of polyhedral cones.

\item False. 
According to Remark 3.4.5, the first part of the proof of Proposition 3.4.3—which establishes that $\mathcal{T}(\overline{z};\mathcal{F})$ contains this specific projection—does not require any constraint qualifications. The CQs are only required for the reverse implication.
\end{enumerate}

\subsection{Solutions for computational questions}
\subsubsection*{Question 1.}
For the tangent cone, the feasible region is $\mathcal{F}^{NLP} = \mathbb{R} \times \{(0,0)\}$. Therefore, $\mathcal{T}(\overline{z}; \mathcal{F}^{NLP}) = \mathbb{R} \times \{(0,0)\}$.
For the linearized cone, evaluating the gradients of the constraints at $\overline{z} = (x,0,0)$ yields $\mathcal{L}(\overline{z}; \mathcal{F}^{NLP}) = \mathbb{R} \times \{(dy, dy) : dy \ge 0\}$.
Since $\mathcal{T}(\overline{z}; \mathcal{F}^{NLP}) \subsetneq \mathcal{L}(\overline{z}; \mathcal{F}^{NLP})$, the basic constraint qualification fails.

\subsubsection*{Question 2.}
For the feasible region, $\mathcal{F}$ is the union of three sets based on the complementarity conditions: $\{(x,0,0) : x \ge 0\}$, $\{(x,0,-x^3) : x \le -1\}$, and $\{(x,-x^2-x,-x^3) : x \in [-1,0]\}$.
For the tangent cone, taking the limit of $\tau_k^{-1}(z^k - \overline{z})$ as $x \to 0$ isolates the active branches at the origin. Therefore, $\mathcal{T}(\overline{z}; \mathcal{F}) = \{(a,0,0) : a \ge 0\} \cup \{(-a,a,0) : a \ge 0\}$.

\subsubsection*{Question 3.}
For the linearized cone, $\mathcal{L}(\overline{z}; \mathcal{F}^{NLP}) = \{(a,b,c) : a+b \ge 0, b \ge 0, c \ge 0\}$.
The convex hull of the tangent cone from Q2 is $Conv(\mathcal{T}) = \{(a,b,0) : a+b \ge 0, b \ge 0\}$. $Conv(\mathcal{T})$ restricts $c=0$, whereas $\mathcal{L}$ allows $c \ge 0$. Therefore, the convex hull is a proper subset of the linearized cone.

\subsubsection*{Question 4.}
The VI gradient equation is $(y-x) + \lambda_1(2y) + \lambda_2(1,0)^T = 0$. Plugging in $\overline{x}=(2,0)$ and $\overline{y}=(1,0)$ yields $(-1,0)^T + \lambda_1(2,0)^T + \lambda_2(1,0)^T = 0$, or $2\lambda_1 + \lambda_2 = 1$.
The multiplier set is $M(\overline{z}) = \{(\lambda_1, \lambda_2) \ge 0 : 2\lambda_1 + \lambda_2 = 1\}$. 
The extreme points are $M^e(\overline{z}) = \{(1/2, 0), (0, 1)\}$.
Therefore, SMFCQ fails: SMFCQ holds if and only if $M(\overline{z})$ is a singleton, but here it is a line segment.

\subsubsection*{Question 5.}
For the gradients, $\nabla_y g_1 = (2,0)^T$ and $\nabla_y g_2 = (1,0)^T$. $\nabla_x g_i = 0$.
For $\overline{\lambda} = (1/2, 0)$, $\lambda_1 > 0 \implies dy^T \nabla_y g_1 = 0 \implies 2dy_1 = 0 \implies dy_1 = 0$. For $\lambda_2 = 0 \implies dy^T \nabla_y g_2 \le 0 \implies dy_1 \le 0$.
Therefore, $\mathcal{K}(\overline{z}, \overline{\lambda}; dx) = \{(0, dy_2)^T : dy_2 \in \mathbb{R}\} = \{0\} \times \mathbb{R}$.

\subsubsection*{Question 6.}
$\mathcal{F} = \mathbb{R}_+ \times \{0\}$, so $\mathcal{T}(\overline{z}; \mathcal{F}) = \mathbb{R}_+ \times \{0\}$. From the linearizations, $\mathcal{L}(\overline{z}; \mathcal{F}) = \mathbb{R}_+^2$.
Therefore, $\mathcal{T}(\overline{z}; \mathcal{F}) \subsetneq \mathcal{L}(\overline{z}; \mathcal{F})$. Since the tangent cone is a proper subset of the linearized cone, the full CQ fails.

\subsubsection*{Question 7.}
The index set is $\beta(\overline{z}) = \{1, 2\}$ because both $y$ and $F(x,y)$ are $0$ at the origin.
For the stationarity system of $\beta_1=\{1\}$, based on the gradients of the specific subsets, the system reduces to $1 - \pi_1 = 0$, $0 - \pi_2 \ge 0$, and $\pi_1 \ge 0$.
Therefore, $(\pi_1, \pi_2) = (1, 0)$ strictly satisfies this system.

\subsubsection*{Question 8.}
$F(\overline{z}) = y-x = (-1,0)^T$. The objective is $\min -dy_1$.
For the constraints, $dy^T \nabla_y g_1 \le 0 \implies 2dy_1 \le 0$ and $dy^T \nabla_y g_2 \le 0 \implies dy_1 \le 0$. 
For the program $\max dy_1$ subject to $dy_1 \le 0$, the optimal solution is exactly $dy_1 = 0$ with $dy_2$ free, completely independently of $dx$.

\subsubsection*{Question 9.}
For the linear independence, $\nabla_y g = I_{2 \times 2}$, which is linearly independent.
For the multiplier set, solving $F + \nabla_y g^T \lambda = 0$ yields $(1,1)^T + I\lambda = 0 \implies \lambda = (-1,-1)^T$. 
Because $\lambda \not\ge 0$, $M(\overline{z}) = \emptyset$. The point is not a valid KKT point, so strong nondegeneracy definitively fails.

\subsubsection*{Question 10.}
Consider any point $x$ arbitrarily close to $\overline{x}=(2,0)$ where $x_2 \neq 0$. The unique projection onto the circle is $y(x) = (x_1, x_2)/\sqrt{x_1^2+x_2^2}$.
At this point, $y_1(x) < 1$. Thus, $f(x,y) < 1 = f(\overline{x},\overline{y})$.
Because points with strictly lower objective values exist arbitrarily close to $\overline{z}$ for $\lambda^*$, the KKT triple is not a local minimum. This proves the "for all" requirement is essential for local minima equivalence.

\subsection{Solutions for proof questions}
\subsubsection*{Question 1}
(a) Proof Sketch: Let $dz \in \mathcal{T}(\overline{z}; \mathcal{F})$. By definition, $dz = \lim \tau_k^{-1}(z^k - \overline{z})$ for some sequence $z^k \to \overline{z}$. For each $k$, we can select an extreme multiplier $\lambda^k \in M^e(z^k)$. Due to the Constant Rank Constraint Qualification (CRCQ), the sequence of extreme points $\{\lambda^k\}$ will have a limit point $\overline{\lambda}$ that belongs to $M^e(\overline{z})$. By passing to the limit in the linearized equations (Equation 3.2.6), $dz$ solves the AVI for $\overline{\lambda}$, meaning $dz \in \mathcal{L}^e(\overline{z}; \mathcal{F})$. 
(b) Limiting Step: CRCQ is strictly necessary to guarantee that the limiting gradients maintain linear independence. Without CRCQ, the limit of a sequence of extreme multipliers $\lambda^k \in M^e(z^k)$ might end up in the relative interior of $M(\overline{z})$ rather than at an extreme point, which would fail to place the limit in $M^e(\overline{z})$.

\subsubsection*{Question 2}
Proof Sketch: In Example 3.4.2, the extreme multiplier set is $M^e(\overline{z}) = \{(1/2,0), (0,1)\}$. The actual tangent cone $\mathcal{T}(\overline{z}; \mathcal{F})$ is exactly equal to the graph of the solution mapping for just one of these multipliers: $Gr(\mathcal{LS}_{(\overline{x}, \overline{\lambda})})$ where $\overline{\lambda} = (1/2,0)$. 
Failure of Extreme CQ: The cone $\mathcal{L}^e(\overline{z}; \mathcal{F})$ requires taking the union over all extreme points. Because $Gr(\mathcal{LS}_{(\overline{x}, (0,1))})$ contains directions not present in the actual tangent cone, $\mathcal{T}(\overline{z}; \mathcal{F}) \subsetneq \mathcal{L}^e(\overline{z}; \mathcal{F})$. Extreme CQ fails because this is a proper subset. Basic CQ holds because there exists a valid subset $M' = \{(1/2,0)\}$ where equality holds.

\subsubsection*{Question 3}
Derivation: For NCPs, $g(x,y) \equiv -y$. The active set $\mathcal{I}(\overline{z})$ corresponds to $y_i = 0$. Because $M(\overline{z}) = \{F(\overline{z})\}$, the degenerate index set $\mathcal{I}_0(\overline{z}, \lambda)$ contains indices where both $y_i = 0$ and $F_i(\overline{z}) = 0$. This is the exact definition of the set $\beta(\overline{z})$.
Bijection: Theorem 3.3.4 partitions $\mathcal{I}_0$ into subsets $\alpha$ and $\overline{\alpha}$. By mapping $\alpha \mapsto \beta_1$ and $\overline{\alpha} \mapsto \beta(\overline{z}) \setminus \beta_1$, the generic complementarity conditions directly collapse into the $2^{|\beta(\overline{z})|}$ disjunctive linear systems shown in Theorem 3.3.6. 

\subsubsection*{Question 4}
(a) Counter-example Logic: If only MFCQ holds, $M(\overline{z})$ is a compact polytope but not a singleton. In the proof of Proposition 3.4.3, we must bound the sequence $\tau_k^{-1}(\lambda^k - \overline{\lambda})$. If SMFCQ fails, this sequence can diverge.
(b) Geometric Failure: Without SMFCQ, the normalized direction sequence $\overline{d} = \lim (\lambda^k - \overline{\lambda}) / \|\lambda^k - \overline{\lambda}\|$ might be non-zero while satisfying $\sum \overline{d}_i \nabla_y g_i(\overline{z}) = 0$. This allows the KKT tangent directions to contain unbounded multiplier components that do not project onto valid primal directions in $\mathcal{T}(\overline{z}; \mathcal{F})$, making the projection a strict superset.

\subsubsection*{Question 5}
Proof Sketch: By Proposition 3.3.8, strong nondegeneracy means active gradients are linearly independent and strict complementarity holds ($\overline{\lambda} - g(\overline{x},\overline{y}) > 0$). This implies $\mathcal{I}_0(\overline{z}, \overline{\lambda}) = \emptyset$. 
Collapse: Because the degenerate set is empty, the power set has only one element ($\emptyset$). The directional critical set $\mathcal{K}$ transitions from a polyhedral cone to a linear subspace (an affine subspace for a given $dx$). The complementarity multiplier vanishes from the MPEC Lagrangian because the active constraints are strictly binding, effectively acting as standard equality constraints locally and smoothing the feasible region into a differentiable manifold.

\subsubsection*{Question 6}
Conclusion: The "for all" requirement is no longer sufficient.
Proof Sketch: Proposition 3.4.1 relies critically on the convexity of $g_i(x,\cdot)$ to guarantee that a KKT neighborhood $W \times \mathbb{R}^l$ maps cleanly to a primal feasible neighborhood $W \cap \mathcal{F}$. If $C(x)$ is non-convex, a point $(\overline{x}, \overline{y}, \overline{\lambda})$ could be a local minimum of the KKT NLP formulation (acting as a stationary point or saddle in the primal space), but $(\overline{x}, \overline{y})$ may fail to be a valid local minimizer of the true MPEC because the primal feasible space lacks the convexity required to bound the neighborhood objective values.

\section{Bibliographic remarks and Acknowledgment}

\textbf{This note is mainly based on Chapter 3, in the MPEC monograph of Zhi-Quan Luo, Jong-Shi Pang and Daniel Ralph.} Please see the relevant references: 
\cite{falk1995bilevel,
aghasi2025fully,
hong2023two,
kovccvara1994optimization,
chaudet2020shape,
liu2023auxiliary,
liu2025bidirectional,
outrata1995numerical,
cui2023complexity,
christof2020nonsmooth,
rawat2026augmented,
robinson1980strongly,
chen2025control,
shin2023near,
bolte2024differentiating,
nghia2025geometric,
wang2025analysis,
chen2025aubin,
liu2024auxiliary,
kojima1980strongly,
chen2026characterizations,
cui2026lipschitz,
shin2022exponential,
de2023function,
bank1982d,
wang2026algebraic,
bonnans1994local,
khanh2024globally,
chen2025two,
mohammadi2022variational,
dussault2026polyhedral,
gowda1994stability,
liu2025learning,
qi2000constant,
facchinei1998accurate,
jittorntrum2009solution,
kyparisis1992parametric,
liu1995sensitivity,
pang11995stability,
liu2025risk,
liu2024feasibility,
qiu1992sensitivity,
aussel2024variational,
reinoza1985strong,
facchinei2003finite,
harker1990finite,
kyparisis1990sensitivity,
wang2025multi,
kleinmichel1972av,
ortega2000iterative,
wang2025analysis1,
bai2021matrix,
lin2026hierarchical,
gao2022rolling,
gander2026landmarks,
mishchenko2023regularized,
han2025low,
wang2022newton,
robinson2009generalized,
mordukhovich2023globally,
wang2026damage,
robinson2009local,
shuo2026lecture,lin2023monotone,yu2026optimization,liang2025squared,jongen1987inertia,guddat1990parametric,bellon2024time,tang2022running,liu2025new,josephy1979newton,si2024riemannian,yu2026optimization1,longman2023method,yao2023relative,han2024continuous,ha1987application,yu2026pattern,kojima2009continuous,seguin2022continuation,liu1995perturbation,pang1996piecewise,liu2022iterative,eikenbroek2022improving,ralph1995directional,yu2026optimization2,scheel2000mathematical,bonnans1992developpement,bonnans1992expansion,dempe1993directional,liu2023iterative,shapiro1988sensitivity,pang1990newton,robinson1991implicit,zheng2025enhanced,hang2025smoothness,hang2024role,scholtes2012introduction,cui2022nonconvex}. 

\bibliography{reference1} 

\end{document}